\documentclass[11pt]{amsart}

\usepackage{mystyle}

\begin{document}

\title{Shape Theorems for Poisson Hail\\ on a Bivariate Ground}

\author{Fran\c cois Baccelli}
\address{Department of Mathematics, University of Texas, Austin, TX 78712}
\email{baccelli@math.utexas.edu}

\author{H\'ector A. Chang-Lara}
\address{Department of Mathematics, Columbia University}
\email{changlara@math.columbia.edu}

\author{Sergey Foss}
\address{Heriot-Watt University, Edinburgh, EH14 4AS, UK
and Sobolev Institute of Mathematics, Novosibirsk, 630090, Russia}
\email{s.foss@hw.ac.uk}

\begin{abstract}

We consider an extension of the Poisson Hail model where the service speed is either zero or infinity at each point of the Euclidean space.
We use and develop tools pertaining to sub-additive ergodic theory
in order to establish shape theorems for the growth of the ice-heap
under light tail assumptions on the hailstone characteristics.
The asymptotic shape depends on the statistics of the
hailstones, the intensity of the underlying Poisson point
process and on the geometrical properties of the zero speed set.
\end{abstract}
\maketitle


\section{Introduction}

The present paper revisits the Poisson Hail growth model introduced in \cite{MR2865637}. This model features i.i.d. pairs, consisting each of a
compact Random Closed Set (RACS) and a positive number, arriving on $\R^d$
according to a Poisson rain. Each pair is referred to as a {\em hailstone};
the RACS is referred to as the {\em footprint} of the hailstone and 
the positive number is its {\em height}.
Each point of the Euclidean space is a server (in the queuing theory sense).
The case studied in \cite{MR2865637} is that with one type of servers.

The pure growth model is that where the service
speed of each point of $\R^d$ is zero, and where the hailstones accumulate
over time to form a random (ice) heap.
This model can be seen as a simplified version of the  
the so-called {\em diffusion limited aggregation} (DLA)
model \cite{Halsey} with half space initial condition.  
The main difference between this model and DLA lies in the fact that
the hailstones fall in a privileged direction (e.g. according to gravity)
in the former case rather than in a diffusive way in the latter.

The height of a tagged hailstone in this heap is the sum of its own height plus the maximum of the heights of all hailstones that arrived before and that have a footprint that intersects that of the tagged one.
It was shown in \cite{MR2865637} that when the $d$-th power of the random diameters and the random heights have light-tailed distributions, i.e. have finite exponential moment, then the growth of the random heap of the pure growth model
is asymptotically linear with time. 
This result was extended to certain heavy tailed distributions lately \cite{FossMount}. In \cite{MR2865637}, the case where all servers have a constant
positive service speed was also analyzed.
The model with positive service speed is motivated by
wireless communications: transmitters arrive according
to a Poisson rain in the Euclidean plane.
The footprint of an arrival is a spatial exclusion area
which should be free of other transmitters during some random transmission
time (the height of the arrival).
The hard exclusion rule is simply obtained by a First in First out
serialization:
an arriving transmitter should first wait for his exclusion
area to be free of all those arrived before; it then transmits and
finally leaves.

The present paper considers a bivariate generalization of this model
with two types of servers.
All points in some subset of $\R^d$, called the {\em substrate}, have
zero service speed,
whereas service speed is infinite in the complement. 
For instance, when the substrate is limited to a single point (a special
case that we refer to as the {\em stick model} below),
hailstones get aggregated to the heap if their footprint intersects
this point or the footprint of any earlier hailstone that is part of
the heap, which is some analogue of DLA with an initial condition
given by a point. As above, the main difference is that the
diffusive and isotropic arrivals of DLA are replaced by pure gravitation.
In the wireless setting alluded to above, 
this model allows one to evaluate the negative consequences of
the FIFO rule. The substrate represents a customer
with a very long transmission time (zero speed) and the
complement normal operation (simplified to infinite speed).
The bivariate model hence explains how congestion builds at the fluid
scale in this FIFO model.

The present paper studies the asymptotic shape of this RACS when time tends to infinity in this bivariate speed setting. 

Like the model of \cite{MR2865637}, this variant belongs to the class of infinite dimensional max-plus linear systems \cite{baccelli92syncandlin}. Among the few instances of such systems studied in the past, the closest is the work on infinite tandem queuing networks \cite{Baccelli99asymptoticresults}. The underlying structure of the max-plus recursion in \cite{Baccelli99asymptoticresults} is a two dimensional lattice. In contrast, here, the underlying structure of the recursion is random. Among common aspects, let us stress shape theorems. The lattice shape theorems in \cite{Baccelli99asymptoticresults} are related to those in first passage percolation \cite{sepp}, in the theory of lattice animals \cite{CGGK,GaKe}. Those of the present paper pertain to first passage percolation in random media. This topic was studied in certain random graphs like the configuration model \cite{er1} lately. The shape theorems established in the present paper are based on random structures of the Euclidean space, which stem from point process theory (Poisson rain) and stochastic geometry (RACS).

In Section \ref{sec:model} we provide the precise formulation of the model. In Section \ref{sec:stick} we study the stick model alluded to above. In this case, Theorem \ref{thm:onesupp} establishes a linear asymptotic growth for the maximum height of the heap in a convex set of directions. Also for the stick model, Theorem \ref{thm:hor_growth} establishes a linear asymptotic growth for the footprint of the heap. Both proofs rely on the version of the Super Additive Ergodic Theorem by T. Liggett \cite{MR806224}. Based on these results we are able to prove in Theorem \ref{thm:phasetransition} the existence of an asymptotic phase transition for the heap in the stick model.

The stick model is interesting not only because of its similarity with e.g. the DLA, but also because it is instrumental 
to extend some of the previous results to more general substrates,
as shown in the subsequent sections. The idea originates from \cite{MR2865637} and, heuristically, it consists in reversing time and gravitation about a given point. Analogues of Theorem \ref{thm:onesupp} are extended by this duality argument for compact substrates in Section \ref{sec:bounded}, Theorem \ref{thm:bounded} and convex conical substrates in Section \ref{sec:convexCone}, Theorem \ref{thm:cone_base}. Let us emphasize that conic substrates are the basic cases we need to understand after performing the blow-up of a given profile which arises in the asymptotic analysis. The extension of these results to non convex conical substrates remains open.

\section{The model}\label{sec:model}

We consider a queue where the servers are the points of $\R^d$. We distinguish two types of servers: $\cC$  is the set servers with a service speed equal to zero, and $\R^d \sm \cC$ is that of servers with a service speed equal to infinity. The customers are characterized by:
\begin{enumerate}
\item A random closed set (RACS) of $\R^d$, such that the $d$-th power of the
diameter has a light-tailed distribution;
\item A random service time also light-tailed.
\end{enumerate}
These customers arrive to the queue ($\R^d$) according to a Poisson rain with intensity $\l$.

Starting with an empty queue at time $t=0$, a customer gets queued if it hits $\cC$ or if it hits an earlier customer which was already queued.

%
%
\begin{figure}
 \begin{center}
  \includegraphics[width=8cm]{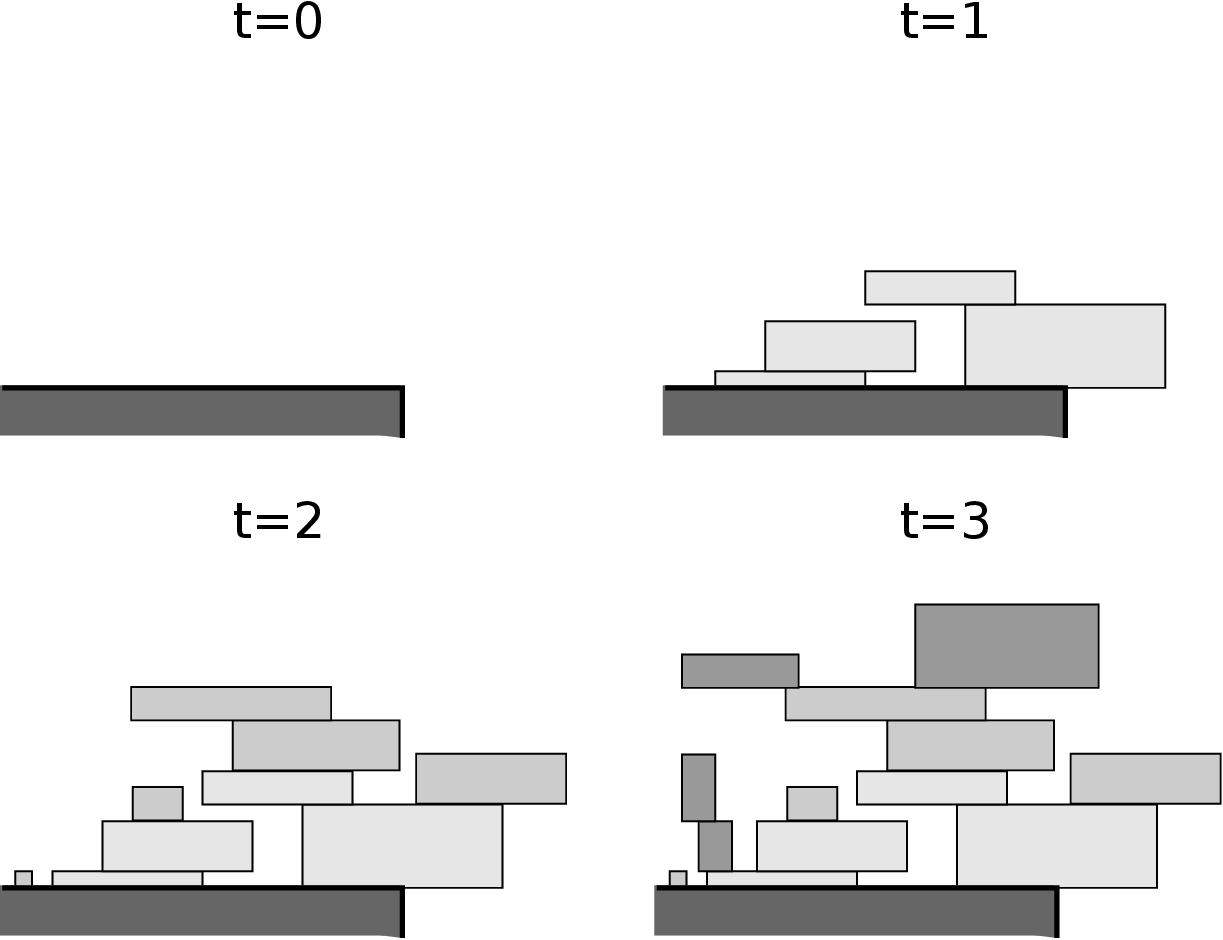}
 \end{center}
 \caption{Evolution of a heap}
\end{figure}

The ice heap is a random set of $\R^d\times \R$, and the main questions of interest are about the growth of its height in various directions, and about the growth of its {\em spatial projection} 
(defined as its projection on $\R^d$), again in various directions.

\subsection{Precise Formulation}

Consider a homogeneous Poisson point process $\Phi$ in $\R^d\times\R$ with intensity $\l>0$ defined on a probability space $(\W,\cF,\P)$. $\Phi$ can be seen as simple counting measure, namely as a sum of delta distributions at (different) points in $\R^d\times\R$. For every Borel set $A \ss \R^d\times\R$, $\Phi(A)$ counts the number of points that belong to the set $A$. By being Poisson homogeneous we mean the following:
\begin{enumerate}
\item $\Phi(A)$ has a Poisson distribution with parameter $\l|A|$,
where $|\cdot|$ denotes the Lebesgue measure in $\R^d\times\R$.
\item Given pairwise disjoint subsets $A_1, \ldots, A_n$ of $\R^d\times\R$, the random variables $\Phi(A_1), \ldots, \Phi(A_n)$ are independent.
\end{enumerate}

This point process is independently marked. Each point comes with a pair of marks. These pairs are independent and identically distributed. However stochastic dependence within a pair is allowed. Let $\{(C_{(x,t)},\s_{(x,t)})\}_{(x,t) \in \Phi}$ denote the marks. These are i.i.d. random pairs. The mark of point $(x,t)$ consists of a
compact RACS $C_{(x,t)}$ centered at the origin (e.g. the center of mass of the RACS
is 0) and of a random variable $\s_{(x,t)}$ taking values in $\R^+$. 

Let 
$$
\xi_{(x,t)} := \diam (C_{(x,t)}) := \sup \{ |y-z|: y,z\in C_{(x,t)} \}
$$ 
be the diameter of set $C_{(x,t)}$. We assume that both random variables 
$\sigma_{(x,t)}$ and $\xi^d_{(x,t)}$ (the $d$th power of $\xi_{(x,t)}$)
are {\it light-tailed}, in that
%
\begin{align}
\E(\exp(c\xi_{(x,t)}^d)) < \8,\quad
\E(\exp(c\s_{(x,t)})) < \8,
\label{eq:ltail}
\end{align}
for some constant $c>0$ (note that the law of $\xi_{(x,t)}$
is the same for all $(x,t)$ 
and that a similar observation holds for $\s_{(x,t)}$;
so that there is only two conditions here).
%

The homogeneity assumption is reflected by the following {\em compatibility property}. Given the group of translations
\begin{align*}
T_{(x_0,t_0)}: (x,t) \mapsto (x,t)+(x_0,t_0)
\end{align*}
of $\R^d\times\R$, there exists $S:\R^d\times\R\times\W \to \W$ measurable and satisfying the following properties:
\begin{enumerate}
 \item {\em Measure preserving:} For every $(x_0,t_0) \in \R^d\times\R$, $S_{(x_0,t_0)}:\W\to\W$ is measure preserving.
 \item {\em Group property:} $S_{(x_0,t_0)} \circ S_{(x_1,t_1)} = S_{(x_0 + x_1,t_0 + t_1)}$ and $S_{(0,0)} = Id$.
 \item {\em Compatibility:}
 \begin{align*}
 \Phi\circ S_{(x_0,t_0)} (A)(\omega) = \Phi(T_{(x_0,t_0)}A)(\omega),\quad A\subset \R^d\times \R.
 \end{align*}
\end{enumerate}
One can then extend the sequence of marks to a random process $(C_{(x,t)},\s_{(x,t)})$ defined on $\R^d\times \R$ and such that
\begin{align*}
(C_{(x,t)},\s_{(x,t)}) = (C_{(0,0)},\s_{(0,0)}) \circ S_{(x,t)},\quad
\forall (x,t).
\end{align*}

Because of the Poisson and independence assumptions, there is no loss of generality
in assuming that the flow $S$ is ergodic. In particular, for every measurable $G \ss \W$ such that
\begin{align*}
\P(S_{(0,t)}^{-1}G \Delta G) = 0 \text{ for every $t \in \R$},
\end{align*}
we have $\P(G) = 0$ or 1. Here $F\Delta G = (F\setminus G) \cup (G\setminus F)$ is the symmetric difference
of $F$ and $G$.

\subsection{Height Profile Function}

Let $H_{(x,t)}$ be the height of the heap at location
$x \in \R^d$ at time $t \geq 0$. When the substrate $\mathcal K$
is the whole Euclidean space, the construction of
this function and the identification of the conditions under which
it is non degenerate (e.g. not equal to $+\infty$ a.s.
for all $x$ and all $t>0$) are one of the main achievements
of \cite{MR2865637}. This construction relies on a sequence of
steps, all relying on the monotonicity properties
of the dynamics. These steps, which include a discretization scheme,
a percolation argument and a branching upper bound, are
combined to show that, under the foregoing tail and
independence assumptions, $H_{(x,t)}$ is a.s. finite
for all $x$ and $t<\infty$.

The tail and independence assumptions are the same as in \cite{MR2865637}.
The finiteness of the height profile function
for a substrate $\mathcal K\subset \R^d$ then
follows from the monotonicity properties of this
function w.r.t. the initial condition which is here
\begin{align*}
H_{(x,0)} = \begin{cases}
0 &\text{ if $x \in \cC$},\\
-\8 &\text{ if $x \notin \cC$}
\end{cases}
\end{align*}
in place of $H_{(x,0)}\equiv 0$ in \cite{MR2865637}.
The construction of \cite{MR2865637} also shows that,
for all $x$, the function $t\to H(x,t)$ is piecewise constant.
Note that it here takes its values
in $\R^+\cup \{-\infty\}$. It will be assumed right continuous.
The left limit of $H_{(x,.)}$ at $t$ will be denoted by $H_{(x,t-)}$

\subsection{Stochastic Differential Equation}
The dynamics can also be described by a stochastic differential equation 
which we briefly outline in this subsection (in spite of the
fact that it will not be used below) as it is of independent interest.

Let $N^x$ denote the Poisson point process of $\R^d\times \R$ of
RACS intersecting location $x$, i.e.
\begin{align*}
N^x(B\times [a,b]) = \int_{B \times [a,b]} 1(x \in C_{(y,s)} + y) \Phi(dyds),
\end{align*}
for all $a<b$ and $B$ Borel sets of $\R^d$.
For $t > u \geq 0$, if $H_{(x,u)}>0$, then
\begin{align}\label{eq:recursionH}
H_{(x,t)} = H_{(x,u)} + \int_{\R^d\times [u,t]} \1\s_{(z,v)} +
\sup_{y \in C_{(z,v)} + z} H_{(y,v)} - H_{(x,v-)} \2N^x(dzdv).
\end{align}
The rationale is that at the first point of $N^x$, say $(z,w)$ in $[u,t]$ if any,
$H(x,u)$ is cancelled by $H(x,w-)$ and the new value
of $H(x,.)$ is 
$$H(x,w)=\s_{(z,w)} + \sup_{y \in C_{(z,w)} + z} H_{(y,w)}.$$
If $H_{(x,u)}=-\infty$, this equation still holds when interpreting
$-\8$ as a $-K$ with $K$ large. For instance, in this case, at the first arrival of $N^x$,
\begin{eqnarray*}
H_{(x,t)} & = & H_{(x,u)} + \s_{(z,w)} + \sup_{y \in C_{(z,w)} + z} (H_{(y,w)} - H_{(x,w-)})\\
& = & -K + \s_{(z,w)} + \sup_{y \in C_{(z,w)} + z} H_{(y,w)} + K\\
& = & \s_{(z,w)} + \sup_{y \in C_{(z,w)} + z} H_{(y,w)} .
\end{eqnarray*}
If for all $y\in C_{(z,w)} + z$, $H_{(y,w)}=-\infty$, then $H_{(x,w)}=-\8$ too.
Else, $H_{(x,w)}>0$.

It follows from the construction summarized in the 
previous subsection that, under the foregoing
tail and independence assumptions,
(\ref{eq:recursionH}) has a piecewise constant solution.
All the results of the paper can hence be rephrased as
properties of this stochastic differential equation.


\subsection{Monotonicity}
\label{sec:23}
The proposed model is {\it monotone} in several arguments.

{\it Monotonicity  in $\cC$}. For two systems
with the same data $(\Phi, \{C,\s\})$ but with initial substrates
$\cC^{(1)} \ss \cC^{(2)}$, 
the associated heights $H^{(1)}$ and $H^{(2)}$ satisfy
\begin{align*}
H^{(1)}_{(x,t)} \leq H^{(2)}_{(x,t)} \text{ for every $(x,t) \in \R^d\times[0,\8)$}.
\end{align*}

Similarly, there is monotonicity in $t$, in the $\sigma$'s and in the $C$'s.


\section{The Stick Model: $\cC = \{0\}$}\label{sec:stick}

In this Section we consider the case $\cC = \{0\}$ and call it {\it the stick model}.   Theorem \ref{thm:onesupp}
shows that there exists a finite asymptotic limit for the maximal
height of the associated heap $H_{(x,t)}$ (referred to as the
\emph{stick heap} below) in any given convex set of directions.
Theorem \ref{thm:hor_growth} shows that there exists a finite
asymptotic limit for how far the spatial projection of the heap
grows, measured
with respect to a {\em set-gauge} to be defined.

\subsection{Height Growth}
In this section we focus on the maximal height, $\H^{(\Theta)}_t$, of the stick heap among
all directions in a set of directions $\Theta$, which is defined as follows:

\begin{definition}
\label{def31}
For $\Theta \ss S^d_+ := \{(x,h) \in \R^d\times(0,1]: \ |(x,h)| = 1\}$ 
non empty, 
\begin{align*}
\H^{(\Theta)}_t  := \sup\3 h \in [0,\8): \exists \ x \in \R^d \text{ such that } (x,h) \in |(x,h)|\Theta,
H_{(x,t)}\geq h\4.
\end{align*}
\end{definition}
In particular, if $\Theta =\{(0,1)\}$, the north pole of $S^d_+$, then $\H^{(\Theta)}_t=H_{(0,t)}$.

Since $H_{(0,t)}\ge 0$, the set where the supremum is evaluated in the last definition is
non-empty as it always contains $h=0$ (since $0 \Theta=\{(0,0)\}$). This 
also implies that $\H^{(\Theta)}_t \geq 0$.

\begin{figure}
 \begin{center}
  \includegraphics[width=6cm]{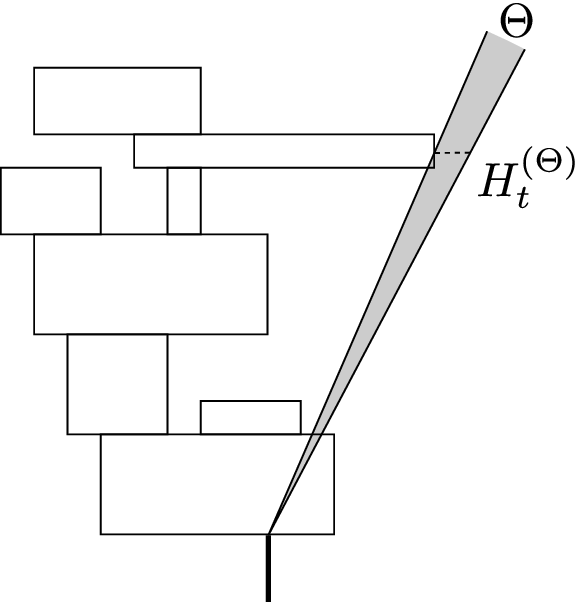}
 \end{center}
 \caption{Definition of $\H^{(\Theta)}_t$ with $\cC = \{0\}$}
\end{figure}

\begin{definition}
A set $\Theta \ss S^d_+$ is convex if for all $\theta_1, \theta_2 \in \Theta$ and $s \in [0,1]$,
\begin{align*}
s\theta_1 + (1-s)\theta_2 \in |s\theta_1 + (1-s)\theta_2|\Theta.
\end{align*}
\end{definition}
Notice that if $\Theta$ is convex, then for all $a,b \geq 0$ and $\theta_1, \theta_2\in\Theta$, we have
\begin{align*}
a\theta_1 + b\theta_2 \in |a\theta_1 + b\theta_2|\Theta.
\end{align*}

\begin{theorem}\label{thm:onesupp}
For all $\Theta \ss S^d_+$ convex and closed, there exists a non-negative constant  $\gamma^{(\Theta)}$ such that
\begin{align*}
\lim_{t\to\8} \frac{\H^{(\Theta)}_t }{t}
= \lim_{t\to\8}\frac{\E \H^{(\Theta)}_t}{t}
= \sup_{t>0} \frac{\E \H^{(\Theta)}_t}{t} = \gamma^{(\Theta)}< \8,
\end{align*}
where the first limit holds both in the a.s. and the $L_1$ sense.
\end{theorem}


Before proving this theorem, we give a few preliminary lemmas.

The following lemma is a direct consequence of the independence of the Poisson rain in disjoint sets and of homogeneity. In this lemma, $\Phi\cap B$ denotes the set of points of $\Phi$ that belong to $B$.


\begin{lemma}\label{lem:aux}
Let $X:\Omega \to \R^d$ be a random variable which is independent of $\{\Phi\cap B,(C_{(y,s)},\s_{(y,s)}): (y,s)\in B\cap \Phi, B \in \mathcal B(\R^d \times (0,\8))\}$.
Then, for every $\Theta \ss S^d_+$, the stochastic process $\{\H^{(\Theta)}_t  \circ S_{(X,0)}, t>0\}$
has the same law as $\{\H^{(\Theta)}_t, t>0\}$ and it is independent of the $\s$-algebra generated by $X$ and $\{\Phi\cap B,(C_{(y,s)},\s_{(y,s)}): (y,s)\in B\cap \Phi, B \in \mathcal B(\R^d \times (-\8,0])\}$.
\end{lemma}

%

\begin{lemma}\label{lem:L1bound}
\begin{align}
\label{eq:suphei}
\sup_{t>0} \frac{\E \H^{(S^d_+)}_t}{t} < \8.
\end{align} 
\end{lemma}
The proof of
Lemma \ref{lem:L1bound} is quite close to that of
Theorem 2 in \cite{MR2865637}. In order to
make the paper self-contained, we 
provide a proof in Appendix. 

\begin{lemma}\label{lemmedebase}
For all $0 \leq t_1 < t_2$ and $x,y \in \R^d$, the stick heap
satisfies the following inequality:
\begin{align*}
H_{(x+y,t_2)} \geq H_{(x,t_1)} + H_{(y,t_2-t_1)}\circ S_{(x,t_1)}.
\end{align*}
\end{lemma}

\begin{proof}
For $t \geq t_1$, let $\tilde H_{(z,t)}$ be constructed 
as explained above
with the initial condition
\begin{align*}
\tilde H_{(z,t_1)} := \begin{cases}
H_{x,t_1} &\text{ if $z = x$},\\
-\8 &\text{ if $z \neq x$}.
\end{cases}
\end{align*}
Then, by monotonicity $H_{(x+y,t)} \geq \tilde H_{(x+y,t)}$
and it suffices to show that 
$\tilde H_{(x+y,t)} = H_{(x,t_1)} + H_{(y,t-t_1)}\circ S_{(x,t_1)}$.


If $H(x,t_1)=-\infty$, then $\tilde{H}(z,t_1)=-\infty$ for all $z$,
and it follows that $\tilde{H}(z,t)=-\infty$ for all $z$
and all $t\ge t_1$. If $H(x,t_1)$ is non-negative, then 
the process $\tilde{H}(z,t_1)$ is nothing else
than the process $H(z,t)$ shifted by $H(x,t_1)$ is space and by $t_1$ in time.
So, in both cases, it satisfies the relation 
$\tilde{H}(x+y,t) = H(x,t_1) + H(y,t-t_1) \circ S(x,t_1)$ indeed. 
\end{proof}

\begin{proof}[Proof of Theorem \ref{thm:onesupp}]
Let $X_t \in \R^d$ be such that,
\begin{align*}
(X_t, \H^{(\Theta)}_t) \in |(X_t, \H^{(\Theta)}_t)|\Theta, \quad
H_{(X_t,t)} \geq \H^{(\Theta)}_t .
\end{align*}
The existence of such an $X_t$ is obtained from 
the proof of Corollary 1 in \cite{MR2865637}.

This proof shows that at time $t$, not only the height, but also the diameter of the heap is a.s. finite\footnote{Later on we will also prove an upper bound for this diameter in Lemma \ref{lem:upper_diam}}. Therefore, with probability 1, one can find 
at least one $X_t$ that satisfies the above properties. There could be more than 
one and, in order for $X_t$ to be a random variable (i.e. a measurable function), we may, 
for instance, take the smallest $X_t$ in the lexicographical order.

For $0 \leq t_1 \leq t_2$, let 
\begin{align*}
\H^{(\Theta)}_{t_1,t_2} := \H^{(\Theta)}_{t_1-t_2} \circ S_{(X_{t_1},t_1)}.
\end{align*} 
In order to prove that the limit in the theorem exists and is a.s. constant,
we use the Super-additive Ergodic Theorem of Liggett, see \cite{MR806224}.
We have to verify that the following properties hold:
\begin{enumerate}
\item Super-additivity: For $t_2 > t_1 \geq 0$
\begin{align*}
\H^{(\Theta)}_{t_2} \geq \H^{(\Theta)}_{t_1} + \H^{(\Theta)}_{t_1,t_2}.
\end{align*}
\item For $t_2 > t_1 \geq 0$, the joint distribution of $\{\H^{(\Theta)}_{t_2,t_2+k}, k > 0\}$ is
the same as that of $\{\H^{(\Theta)}_{t_1,t_1+k}, k > 0\}$.
\item For $k > 0$, $\{\H^{(\Theta)}_{nk,(n+1)k}, n > 0\}$ is a stationary process.
\item Bound for the expectation:
\begin{align*}
\sup_{t>0} \frac{\E \H^{(\Theta)}_t}{t} < \8.
\end{align*}
\end{enumerate}
To prove (1), let $t_2 > t_1 \geq 0$ be fixed and let
\begin{align*}
V = \{(x,h) \in \R^d\times(0,\8): (x,h) \in |(x,h)|\Theta, h \leq H_{(x,t_2-t_1)}\circ S_{(X_{t_1},t_1)}\}
\end{align*}
For $(x,h) \in V$ we have by the convexity of $\Theta$ that,
\begin{align}\label{eq:1}
(X_{t_1}+x,\H^{(\Theta)}_{t_1}+h) \in |(X_{t_1}+x,\H^{(\Theta)}_{t_1}+h)|\Theta.
\end{align}
Moreover,
\begin{align}\label{eq:2}
\H^{(\Theta)}_{t_1}+h &\leq H_{(X_{t_1},t_1)} + H_{(x,t_2-t_1)}\circ S_{(X_{t_1},t_1)},\\
\nonumber &\leq H_{(X_{t_1}+x,t_2)},
\end{align}
where we used Lemma \ref{lemmedebase} in the last inequality.
By combining \eqref{eq:1} and \eqref{eq:2} we get that $\H^{(\Theta)}_{t_2} \geq \H^{(\Theta)}_{t_1} + h$, which implies the super-additive inequality after taking the supremum of $h$ over $(x,h) \in V$.

To prove (2) we go back to the definition of $\H^{(\Theta)}_{t_i,t_i+k}$ for $i = 1,2$,
\begin{align*}
\{\H^{(\Theta)}_{t_i,t_i+k}, k > 0\} &= \{\H^{(\Theta)}_k\circ S_{(X_{t_i},t_i)}, k > 0\}.
\end{align*}
By Lemma \ref{lem:aux} we get that both families of random variables have the same joint distribution as $\{\H_k^{(\Theta)}, k>0\}$.

To prove (3) it is enough to check that, for $k > 0$ fixed,
the random variables $\{\H^{(\Theta)}_{nk,(n+1)k}, n > 0\}$
are identically distributed and independent. By definition,
\begin{align*}
\H^{(\Theta)}_{nk,(n+1)k} = \H^{(\Theta)}_k \circ S_{(X_{nk},nk)}.
\end{align*}
Using Lemma \ref{lem:aux} once again, we get that $\H^{(\Theta)}_{nk,(n+1)k}$
is distributed as $\H^{(\Theta)}_k$. Then the independence property follows
again from Lemma \ref{lem:aux}.

Finally, (4) results from the upper bound given by Lemma \ref{lem:L1bound}.
\end{proof}

\subsection{Spatial Projection}

\begin{definition}
For $t\geq 0$, let $F_t$ be the spatial projection of
the heap, namely the  RACS of $\R^d$ which is
the union of all the RACS added to the heap up to time $t$:
\begin{align*}
F_t := \{x\in \R^d: H_{(x,t)} \geq 0\}.
\end{align*}
\end{definition}
If the sets $C(x,t)$ are a.s. connected, so is $F_t$. 
However, if the sets $C(x,t)$ are a.s. convex, $F_t$ has no reason to be convex.

In general, $F_t$ is not necessarily a RACS. However
under the light-tailedness assumptions (\ref{eq:ltail}):
\begin{lemma}\label{lem:upper_diam}
For all finite $t$, $F_t$ is a RACS and
\begin{align}
\sup_{t>0} \frac{\E\1\diam(F_t)\2}{t} < \8.
\label{eq:supdiam}
\end{align} 
\end{lemma}

\begin{proof}
The proof is an application of Lemma \ref{lem:L1bound}, which follows the ideas in the proof of Corollary 1 in \cite{MR2865637}.

$F_t$ is a RACS as a consequence of the upper bound branching process constructed for $F_t$ in the proof of Lemma \ref{lem:L1bound}. This branching process has a.s. finitely many offspring in each generation. This implies that all finite $t>0$, only a finite number of RACS $C_{(x,s)}$ may contribute to $F_t$.

We now prove (\ref{eq:supdiam}). First notice that the set $F_t$ does not depend on the heights.
However we will make use of them in the following way. Assume $\s_{(x,t)} = \xi_{(x,t)} = \diam(C_{(x,t)})$.
%
We now show that under this assumption,
$$4\sup_{x \in \R^d} H_{(x,t)} \geq \diam(F_t).$$
For every $x \in \R^d$ such that $H_{(x,t)} \geq 0$, there exists an integer $n$ and
some set of points $(x_1,t_1), \ldots, (x_n,t_n) \in \R^d \times [0,t)$ such that:
\begin{enumerate}
\item $(x_i,t_i) \in \supp \Phi$ for $i = 1, \ldots, n$;
\item $0 \leq t_i < t_{i+1} < t$ for $i = 1, \ldots, (n-1)$;
\item $x \in x_n+C_{(x_n,t_n)}$ and $H_{(x,s)} = H_{(x, t_n)}$ for $s \in [t_n,t)$;
\item for $i = 1, \ldots, (n-1)$,
there exists $y_i \in x_{i+1}+C_{(x_{i+1},t_{i+1})} \cap x_i+ C_{(x_i,t_i)}$ such that
$H_{(y_i,s)} = H_{(y_i,t_i)}$ for $s \in [t_i,t_{t+1})$; 
\item $0 \in x_1+ C_{(x_1,t_1)}$ and $H_{(0,s)} = 0$ for $s \in [0,t_1)$.
\end{enumerate}
Therefore,
\begin{align*}
|x| \leq |x-x_n| + \sum_{i=1}^{n-1}|x_{i+1} - x_i| + |x_1|
\leq 2\sum_{i=1}^n \diam(C_{(x_i,t_i)})
= 2H_{(x,t)}.
\end{align*}
Maximizing over $\{x \in \R^d:H_{(x,t)} \geq 0\}$ and 
applying Lemma \ref{lem:L1bound} concludes the proof.
\end{proof}

\begin{definition}
Given a direction $v \in S^{d-1}$ and a closed set $A \ss \R^d$,
containing the origin, let
\begin{align*}
D^{(A,v)}_t := \inf \3r \in [0,\8): (A+rv) \cap F_t = \emptyset\4.
\end{align*}
where the infimum of an empty set is $\infty$.
\end{definition}
%
%
%
%

Here are a few examples: If $A = \{0\}$ then $D^{(A,v)}_t$ can be interpreted as the
\textit{internal growth of $F_t$ in the $v$ direction at time $t$}. It is also the contact distance with free space in the $v$-direction. Other interesting cases arise when $A = \{x \in \R^d : x \cdot v \geq 0\}$ or $A = \{x \in \R^d: x = \a v, \a \geq 0\}$; then $D^{(A,v)}_t$ can be interpreted as the
\textit{external growth of $F_t$ in the $v$ direction}. These cases are covered in Theorem \ref{thm:hor_growth} and illustrated in Figure \ref{fig-gau}.

\begin{figure}
 \begin{center}
  \includegraphics[width=10cm]{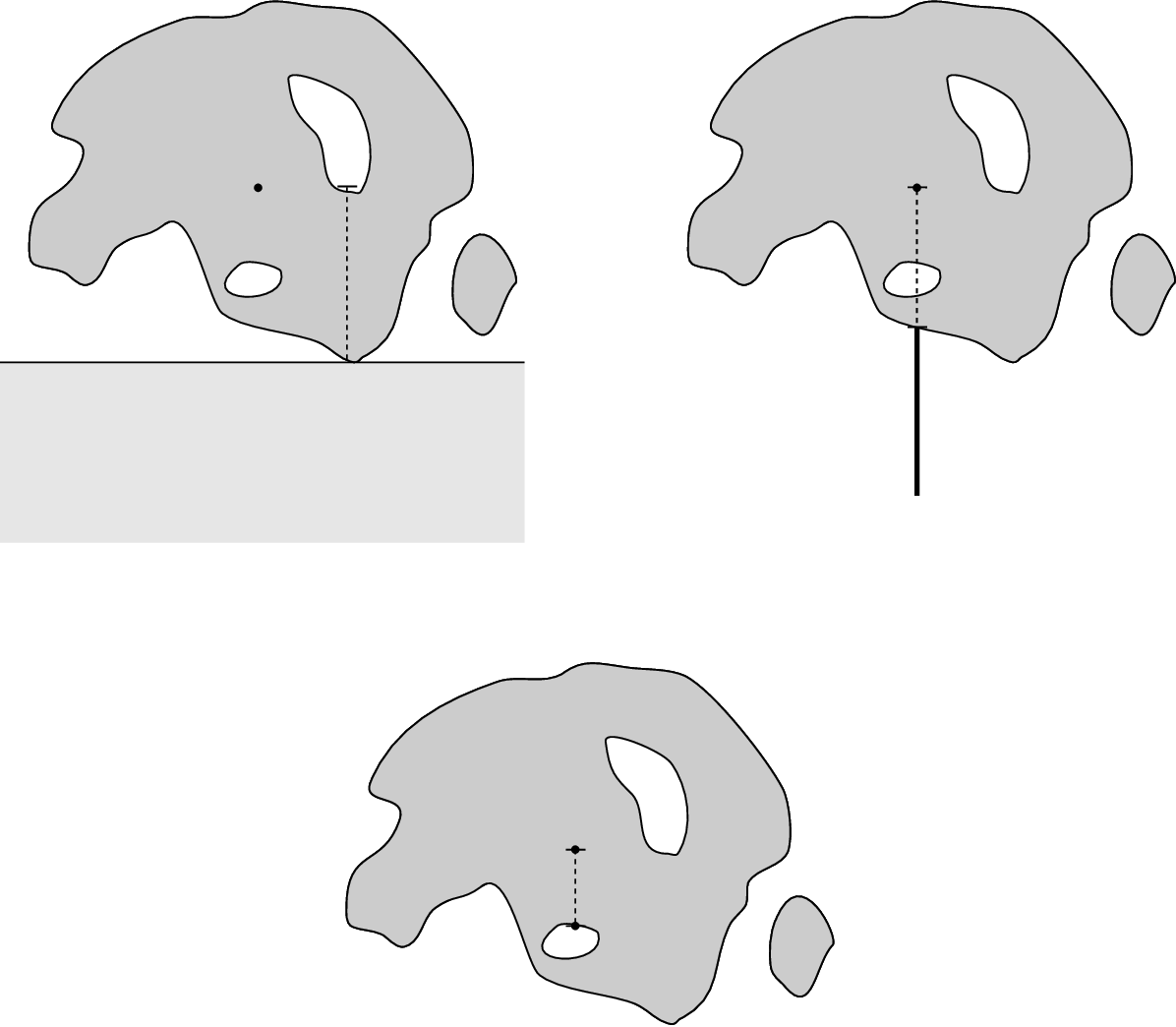}
 \end{center}
 \caption{Different set-gauges measuring the spatial growth of $F_t$.
The direction of $v$ is South. The top-left case is
 $A= \{x \in \R^d : x \cdot v \geq 0\}$; the top-right case is $A = \{x \in \R^d: x = \a v, \a \geq 0\}$; the
bottom case is $A = \{0\}$.}
\label{fig-gau}
\end{figure}

\begin{definition}
The pair $(v,A)$, where
$v \in S^{d-1}$ is a direction and $A \ss \R^d$ a closed set, forms a {\em set-gauge} if
\begin{enumerate}
\item $A$ contains the origin and for every $a \in A$, $A+a \ss A$,
\item $-v$ does not belong to the closed convex hull of $A$.
\end{enumerate}
\end{definition}

The three above examples are set-gauges. Here are other
examples:

If $A$ is a closed convex cone of $\R^d$, different from $\R^d$,
and $-v \notin A$, then $(v,A)$ forms a set-gauge.

If $(v,A)$ forms a set-gauge, then $(v,B)$, where 
$B := \bigcup_{r>0} (A+rv)$ also forms a set-gauge. In this case 
\begin{align*}
D^{(B,v)}_t = \sup \3r \in [0,\8): (A+rv) \cap F_t \neq \emptyset\4.
\end{align*}

Note that for all set-gauges $(v,A)$, $D^{(A,v)}_t$ is a.s. finite. This follows from the property that $F_t$ is a.s. compact and the assumption that $-v$ does not belong to the convex hull of $A$.

Our main result is:
\begin{theorem}\label{thm:hor_growth}
Given a direction $v \in S^{d-1}$ and a closed set $A \ss \R^d$, such that 
$(v,A)$ forms a set-gauge, there exists a non negative constant
$\phi= \phi^{A,v}$ such that
\begin{align*}
\lim_{t\to\8}\frac{D^{(A,v)}_t}{t}
= \lim_{t\to\8}\frac{\E D^{(A,v)}_t}{t}
= \sup_{t>0} \frac{\E D^{(A,v)}_t}{t}=: \phi < \8,
\end{align*}
where the first limit is both a.s. and in $L_1$.
\end{theorem}

\begin{proof}
Once again the proof relies on the distributional Super-additive Ergodic Theorem. Let $X_t \in \R^d$ be a random variable such that,
\begin{align*}
X_t \in (A+D^{(A,v)}_t v) \cap F_t.
\end{align*}
The existence of a finite $X_t$ satisfying this relation follows from the fact that $F_t$ is compact and $A$ is closed. It also uses the fact that $-v$ does not belong to the convex hull of $A$. There is no reason to have uniqueness. However, we can use the same construction as in the proof of Theorem \ref{thm:onesupp} to cope with multiple solutions.

For $0 \leq t_1 \leq t_2$, let 
\begin{align*}
D^{(A,v)}_{t_1,t_2} := D^{(A,v)}_{t_2-t_1}\circ S_{X_{t_1},t_1}.
\end{align*}
By Lemma \ref{lem:aux},
properties analogous to properties (2) and (3)
in the proof of Theorem \ref{thm:onesupp} do hold.
We now prove the super-additivity
and the boundedness of the expectations.

In order to prove the super-additive inequality,
it is enough to show that,
for every $r < D^{(A,v)}_{t_1} + D^{(A,v)}_{t_1,t_2}$,
\begin{align}\label{eq:gauge}
\1A+rv\2 \cap F_{t_2} \neq \emptyset.
\end{align}
If $r < D^{(A,v)}_{t_1}$, this follows from the monotonicity of $F_t$
w.r.t. time and from the definition of $D^{(A,v)}_{t_1}$.
Let now $r = D^{(A,v)}_{t_1} + r'$, with
$r' \in \left[0,D^{(A,v)}_{t_1,t_2}\right)$.
From the definition of $F_t$,
\begin{align*}
F_{t_2-t_1}\circ S_{X_{t_1},t_1} + X_{t_1} \ss F_{t_2}.
\end{align*}
From the definition of a set-gauge and the
property $X_{t_1} \in A + D^{(A,v)}_{t_1}v$,
\begin{align*}
A + X_{t_1} + r'v \ss A+rv.
\end{align*}
From the definition of $D^{(A,v)}_{t_1,t_2}$, for $r' < D^{(A,v)}_{t_1,t_2}$,
\begin{align*}
\1A + r'v\2 \cap \1 F_{t_2-t_1}\circ S_{X_{t_1},t_1}\2 \neq \emptyset,
\end{align*}
which implies
\begin{align*}
\1A +  X_{t_1}+ r'v\2 \cap \1 F_{t_2-t_1}\circ S_{X_{t_1},t_1}
+  X_{t_1}\2 \neq \emptyset
\end{align*}
and \eqref{eq:gauge} follows from the last two inclusions.

Now we prove boundedness of expectations. Given that $A$ and $v$ form a gauge there exists a hyperplane given by $P = \{x \in \R^d: x\cdot w = 0\}$, with $w \in S^{d-1}$, that separates $-v$ and $A$, i.e.
\begin{enumerate}
\item $v\cdot w > 0$,
\item $a \cdot w \geq 0$ for every $a \in A$.
\end{enumerate}
Then, letting $A' = \{x \in \R^d: x\cdot w \geq 0\}$, and using the monotonicity inherited from the fact that $A \ss A'$,
\begin{align*}
D^{(A,v)}_t \leq D^{(A',w)}_t \leq \frac{\diam(F_t)}{|v\cdot w|}.
\end{align*}
Finally, applying Lemma \ref{lem:upper_diam}  we get 
\begin{align*}
\limsup_{t\to\8} \frac{\E D^{(A,v)}_t} t 
\leq \frac{
\limsup_{t\to\8}\E\diam(F_t)/t
}{|v\cdot w|}
< \8.
\end{align*}
So the proof is complete.
\end{proof}

Now we focus on the gauges with $A = \{0\}$.
Our aim is to prove that,
under an extra assumption on the RACS,
\begin{align*}
\lim_{t\to\8} \frac{D^{(A,v)}_t}{t} > 0.
\end{align*}
This will in turn imply that, for every $x \in \R^d$,
the time that it takes for $F_t$ 
to hit $x$ is a.s. finite.

\begin{lemma}\label{lem:hit_one}
Assume that the intensity of $\Phi$ is positive and that, with a positive probability,
the footprint has a non-empty interior that contains the origin. Then, for $A = \{0\}$ and $v \in S^{d-1}$, we have 
\begin{align*}
\lim_{t\to\8} \frac{D^{(A,v)}_t}{t} > 0.
\end{align*}
\end{lemma}

%
\begin{proof}
By Theorem \ref{thm:hor_growth},
\begin{align*}
\lim_{t\to\8} \frac{D^{(A,v)}_t}{t} &= \sup_{t>0} \frac{\E D^{(A,v)}_t}{t} \geq \E D^{(A,v)}_1.
\end{align*}
From the lemma conditions, there is a positive $r$ such that, with positive
probability, $C$ contains the ball $B_r$ with radius $r$ centered at the origin.
By the thinning property of the Poisson point process,
we may consider only the Poisson
rain (with a smaller, but positive intensity) with RACS
that include $B_r$. Then, using
the monotonicity mentioned earlier, we may take $C=B_r$.
In the latter case, it is not difficult to see that 
\begin{align*}
\P(D^{(A,v)}_1 > r/2) \geq \P(\Phi(B_{r/2}\times(0,1]) > 0) > 0. 
\end{align*}
Then $\E D^{(A,v)}_1 > 0$, and the result follows. 
\end{proof}

\begin{definition}
Given $K \ss \R^d$ let $\t(K)$ denote the time
it takes for $F_t$ to cover $K$, i.e.
\begin{align*}
\tau(K) := \inf\{t \in [0,\8]: K \sm F_t 
=\emptyset\}.
\end{align*}
\end{definition}

\begin{remark}\label{rmk:stoping_time}
$\t(K)$ is a stopping time in the sense that $\{\t(K) \leq t\}$ belongs to the $\s$-algebra generated by
\begin{align*}
\{\Phi\cap B,(C_{(y,s)},\s_{(y,s)}): (y,s)\in \Phi\cap B, B \in \mathcal B(\R^d \times [0,t])\}.
\end{align*}
Moreover, $\{\t(K)\leq t\}$ is independent of the $\s$-algebra of subsequent events generated by
\begin{align*}
\{\Phi\cap B,(C_{(y,s)},\s_{(y,s)}): (y,s)\in \Phi\cap B, B \in \mathcal B(\R^d \times [t,\8))\}.
\end{align*}
\end{remark}

\begin{corollary}\label{cor:eventual_cover}
Assume that the intensity of $\Phi$ is positive and the with a positive probability, the footprint has a non-empty interior that contains the origin. Then, for all bounded sets $K \ss \R^d$, $\tau (K)$ is a.s. finite.
\end{corollary}

\begin{proof}
It suffices to assume, by the same reasoning as in the previous proof, that the footprint $C$ is a ball of fixed radius $r>0$, sufficiently small. Let $\{x_1, \ldots, x_n\} \ss K \sm \{0\}$ such that,
\begin{align*}
K \ss \bigcup_{i=1}^n B_{r/2}(x_i).
\end{align*}
Denote also $v_i = \frac{x_i}{|x_i|}$ for $i=1,\ldots,n$.

Consider now $C' := B_{r/2}$ and $F'_t$, $D'^{(A,v)}_t$ constructed from $C'$ and $A = \{0\}$. By Lemma \ref{lem:hit_one} we have that for $i=1,\ldots,n$,
\begin{align*}
\tau'_i := \inf\{t \in [0,\8]: D'^{(A,v_i)}_t \geq |x_i|\} < \8, \quad
\text{ a.s.}
\end{align*}
By construction of $C'$, if $x_i \in F'_t$, then also $B_{r/2}(x_i) \ss F_t$,
which concludes the proof.
\end{proof}

\subsection{Phase Transition}

From Theorem \ref{thm:onesupp} there exists a growth rate $\gamma_{\theta_w(\phi)}$ in the direction,
\[
\theta_w(\phi) := \sin(\phi) e_{d+1} + \cos(\phi)w \in S^d_+, \quad w\in S^{d-1}, \quad \phi\in[0,\pi/2].
\]

For fixed $w\in S^{d-1}$, $\gamma_{\theta_w(\cdot)}$ admits the following phase transition:

\begin{theorem}\label{thm:phasetransition}
Assume that with a positive probability, the footprint has a non-empty interior that contains the origin. Then, for $w\in S^{d-1}$ there exists an angle, $\phi_*(w) \in (0,\pi/2)$ such that $\gamma_{\theta_w(\phi)}$ is positive for any $\phi \in (\phi_*(w), \pi/2]$ and $\gamma_{\theta_w(\phi)}=0$ for any $\phi \in [0,\phi_*(w)).$
\end{theorem}

\begin{proof}
Let us first show that if $\gamma_{\theta_w(\phi)}>0$, then for all $\widehat \phi \in(\phi, \pi/2)$, $\gamma_{\theta_w(\widehat \phi)}>0$.

For any $t>0$, let $x_t$ be defined by
$$t= x_t \tan\phi$$
and let 
$$ \widehat{t} = x_t \tan \widehat \phi. $$
Then
$$
H_{(x_tw,\widehat t)} \ge H_{(x_tw,t)} + H_{(0, \widehat{t}-t)} \circ
S_{(x_tw, t)},
$$
so that
$$
\gamma_{\theta_w(\widehat \phi)} = \lim \frac{H_{(x_tw,\widehat t)}}{\widehat t}
\ge \gamma_{{\theta}_w(\phi)} \frac{\tan \phi}{\tan \widehat\phi} >0.
$$

Let now
$$
\phi_*(w) = \inf \{\phi \in [0,\pi/2] \ : \ \gamma_{\theta_w(\phi)}>0 \}.
$$
If $\phi=\pi/2$, then $\gamma_{\theta_w(\phi)}>0$.
Hence, $\phi_*(w)$ is well defined.  It follows from
the last monotonicity property that it is the threshold above which
$\gamma_{\theta_w(\phi)}>0$.

It remains to prove that this threshold is non degenerate.

Let us first prove that it is positive.
Let $s_w$ denote the spatial growth rate in direction $w$
and $h$ the vertical growth rate. Both $s_w$
and $h$ are positive and finite.
So if the angle $\phi$ is smaller than $\arctan(h/s_w)>0$,
then $\gamma_{\theta_w(\phi)}=0$.

Let us now show that $\phi_*(w) < \pi/2$.

For all $n\in \N$, let $x_n=(nr/2)w$ (here $r>0$ such that with positive probability, $B_r\ss C$). Let $\Pi_n$ be the Poisson rain of RACS's that contain a ball of radius $r$ centered at $x_n$. Let $t_0\equiv T_0>0$ be the first time or arrival of a RACS of $\Pi_0$ and, for each $n=0,1,\ldots$, let $T_{n+1}=T_n+t_{n+1}$ be the first arrival time after $T_n$ of 
a RACS of $\Pi_{n+1}$. The random variables $t_n$ are i.i.d. exponential with mean, say, $b>0$. Also,
$H_{(x_n,T_n)}$ is not smaller than the sum of $(n+1)$ i.i.d. random variables with distribution $H_{(0,T_0)}$. Since ${\mathbb E} H_{(0,T_0)}>0$, it follows that $\liminf H_{(x_i,T_i)}/T_i >0$ a.s.
Further, $T_i/x_i \to 2b/r <\infty$, so $\phi_*(w) \le \arctan (2b/r) < \pi/2$.
\end{proof}

\section{The Model with $\cC$ Compact}\label{sec:bounded}

In this section we study the growth of the heap starting with a compact substrate $\cC\ss\R^d$, in some convex set of directions $\Theta$.

\subsection{Asymptote at $0 \in \cC$}

In this section we fix $0 \in \cC$. Let $\cC^{(0)} = \{0\}$.
Whenever $\H_t^{(\Theta)}$ is computed with respect to $\cC^{(0)}$
(resp. $\cC$) we denote it by $\H_t^{(\Theta,0)}$ 
(resp. $\H_t^{(\Theta)}$). An analogous notation
is used  for all the other possible constructions.
Given a constant $M \geq 0$, the measure preserving transformation
$S_{(0,M)}$ of $\W$ to itself is denoted by $S_M$.

In the next lemma $\t := \t(\cC)$
denotes the time it takes for $F_t^{(0)}$ to cover the whole set $\cC$.

\begin{lemma}\label{lem:Cbounded1}
For all $\Theta \ss S^d_+$ closed and for all $M,t \geq 0$,
the following inequalities hold on $\{\t \leq M\}$,
\begin{align*}
\H^{(\Theta,0)}_{M+t} \geq \H^{(\Theta)}_t \circ S_M \geq \H^{(\Theta,0)}_t \circ S_M.
\end{align*}
\end{lemma}

\begin{proof}
On $\t \leq M$, $\cC \ss F_M^{(0)}$, so that for every $x \in \R^d$,
$H^{(0)}_{(x,M)} \geq H_{(x,0)} \circ S_M$. This implies
that $H^{(0)}_{(x,M+t)} \geq H_{(x,t)} \circ S_M$ for all $t>0$
by the monotonicity in the construction of $H$.
The left-most inequality then follows.

The right-most inequality is just a consequence of the
monotonicity w.r.t. the initial substrates $\cC^{(0)} \ss \cC$.
\end{proof}

\begin{lemma}\label{lem:Cbounded2}
Under the assumptions of Corollary \ref{cor:eventual_cover}
and Theorem \ref{thm:onesupp}, if $0\in \cC$, we have
\begin{align*}
\lim_{t\to\8} \frac{\H_t^{(\Theta)}}{t} = \lim_{t\to\8}\frac{\E \H_t^{(\Theta)}}{t} = \sup_{t>0} \frac{\E \H_t^{(\Theta,0)}}{t} =\gamma^{(\Theta)} < \8,
\end{align*}
where the first limit is both in the a.s. and the $L_1$ sense.
\end{lemma}

Notice that the rightmost term is the one corresponding to $\cC^{(0)}$;
this tells us that asymptotically, the heaps starting at
$\cC$ or $\cC^{(0)}$ behave similarly in terms of directional shape.

\begin{proof}
By  Lemma \ref{lem:Cbounded1}, 
Theorem \ref{thm:onesupp} and the fact
that $S_M$ is measure preserving, we have 
\begin{align*}
\lim_{t\to\8} \frac{\H_t^{(\Theta)}\circ S_M}{t} 
1_{\{\t\leq M\}}
= \sup_{t>0} \frac{\E \H_t^{(\Theta,0)}}{t} 
1_{\{\t\leq M\}} 
\text{ a.s.}
\end{align*}
Hence
\begin{eqnarray*}
\P \left(
\lim_{t\to\8} \frac{\H_t^{(\Theta)}}{t} = \sup_{t>0} \frac{\E 
\H_t^{(\Theta,0)}}{t}\right)
 & = & \P \left(\lim_{t\to\8} 
\frac{\H_t^{(\Theta)}\circ S_M}{t} = \sup_{t>0} \frac{\E 
\H_t^{(\Theta,0)}}{t}\right)\\
& \ge & \P (\tau \le M).
\end{eqnarray*}
Since $M>0$ is arbitrary and $\tau$ is finite a.s., we get
the a.s. convergence of
$\H_t^{(\Theta)}/t$ to the announced limit. 

Now we proceed to show the convergence in $L_1$.
By Lemma \ref{lem:Cbounded1},
\begin{align*}
\1\frac{\H^{(\Theta,0)}_{M+t}}{t}-\sup_{t>0} \frac{\E \H_t^{(\Theta,0)}}{t}\21_{\{\t\leq M\}} &\geq \1\frac{\H^{(\Theta)}_t \circ S_M}{t} -  \sup_{t>0} \frac{\E \H_t^{(\Theta,0)}}{t}\21_{\{\t\leq M\}}.
\end{align*}
Then,
\begin{align*}
&\E\left|\frac{\H^{(\Theta,0)}_{M+t}}{t}-\sup_{t>0} \frac{\E \H_t^{(\Theta,0)}}{t}\right|
\geq \ &\E\left|\1\frac{\H^{(\Theta)}_t \circ S_M}{t} -  \sup_{t>0} \frac{\E \H_t^{(\Theta,0)}}{t}\21_{\{\t\leq M\}}\right|.
\end{align*}
By the independence property
given in Remark \ref{rmk:stoping_time},
and using again that $S_M$ is measure preserving,
\begin{align*}
&\E\left|\frac{\H^{(\Theta,0)}_{M+t}}{t}-\sup_{t>0} \frac{\E \H_t^{(\Theta,0)}}{t}\right| 
\geq \ &\E\left|\frac{\H^{(\Theta)}_t}{t} -  \sup_{t>0} \frac{\E \H_t^{(\Theta,0)}}{t}\right|\P(\t\leq M).
\end{align*}
Choose $M$ sufficiently large so $\P(\t\leq M) > 0$. 
Then letting $t\to\8$ concludes the proof thanks to the $L^1$ convergence
of Theorem \ref{thm:onesupp}.
\end{proof}

\subsection{Asymptote at $0 \notin \cC$}

In this section we assume that $0 \notin \cC\neq \emptyset$.
We use the following notation: $\cC^{(0)} = \{0\}$ and 
$\cC^{(1)} = \cC \cup \{0\}$.
Whenever $\H_t^{(\Theta)}$ is computed with respect to
$\cC^{(0)}$, (resp. $\cC^{(1)}$ or $\cC$) we denote it by
$\H_t^{(\Theta,0)}$ (resp. $\H_t^{(\Theta,1)}$ or $\H_t^{(\Theta)}$).
We use analogous notation for all other possible constructions.

In the next lemma $\t := \t(\cC^{(0)})$ is the time it takes for
$F_t$ to hit the origin.

\begin{lemma}\label{lem:Cbounded3}
For all closed $\Theta \ss S^d_+$ and all $M,t \geq 0$,
the following inequalities hold on $\{\t \leq M\}$,
\begin{align*}
\H^{(\Theta,1)}_{t+M} \geq \H^{(\Theta)}_{t+M} \geq \H^{(\Theta,0)}_t\circ S_M.
\end{align*}
\end{lemma}

\begin{proof}
The proof is very similar to that of Lemma \ref{lem:Cbounded1}.
On $\t \leq M$, $\cC^{(0)} \ss F_M$; therefore for every
$x \in \R^d$, $H_{(M,x)} \geq H_{(0,x)}^{(0)} \circ S_M$, which
implies the second inequality.
%
The first inequality is a consequence of the monotonicity.
\end{proof}

Using this lemma (instead of Lemma \ref{lem:Cbounded1}) and the same ideas as in the proof of Lemma \ref{lem:Cbounded2} gives: 

\begin{theorem}\label{thm:bounded}
Under the assumptions of Corollary \ref{cor:eventual_cover}
and Theorem \ref{thm:onesupp}, in all cases ($0\in \cC$ or $0\notin \cC$),
\begin{align*}
\lim_{t\to\8} \frac{\H_t^{(\Theta)}}{t} = \lim_{t\to\8}\frac{\E \H_t^{(\Theta)}}{t} = \sup_{t>0} \frac{\E \H_t^{(\Theta,0)}}{t} =\gamma^{(\Theta)}< \8,
\end{align*}
with the first limit holding both a.s. and in the $L_1$ sense.
\end{theorem}

\section{The Model with $\cC$ a Convex Cone and its Generalizations}\label{sec:convexCone}

In this section, the substrate is first a convex cone of $\R^d$
with its vertex at the origin,
and then an object similar to such a cone but more general.

\begin{definition}
Given $\cone \ss \R^d$ a closed convex cone
with vertex at the origin, we
define $\Theta(\cone) \ss S^d_+$ to be the following subset of $S^d_+$
(see Definition \ref{def31}):
\begin{align*}
\Theta(\cone) := S^d_+ \cap (\cone\times\R).
\end{align*}
\end{definition}

For the proofs of this section, we use yet another
property of the model, which is some form
of invariance by time reversal. Consider the reflection
\begin{align*}
R: (x,t) \mapsto (x,-t).
\end{align*}
Because the Poisson rain is invariant in law by $R$,
and because the marks are i.i.d.,
there exists a measure preserving $V:\W \to \W$ which 
is compatible with $R$, i.e.:
\begin{align*}
(\Phi\circ V)(A) = \Phi(RA)\quad
C_{(x,t)}\circ V = C_{R(x,t)},\quad
\s_{(x,t)}\circ V = \s_{R(x,t)}.
\end{align*}

In the following theorem,
$\H_t^{(\Theta,0)}$ and $\H_t^{(\Theta,1)}$ are the heights
computed when starting with the substrate $\cC^{(0)} := \{0\}$
or $\cC^{(x)} := \{0,x\}$ respectively whereas $H_{(x,t)}$ is
the height at $x$ when starting with the substrate $\cC=\cone$.

\begin{theorem}\label{thm:cone_base}
Under the assumptions of Corollary \ref{cor:eventual_cover},
for all closed convex cones $\cC \ss \R^d$ 
and all $x \in \R^d$, 
\begin{align}
\lim_{t\to\8} \frac{\max(0,H_{(x,t)})}{t} =
\lim_{t\to\8}\frac{\E\max(0, H_{(x,t)})}{t} = \sup_{t>0}\frac{\E \H_t^{(\Theta(\cC),0)}}{t}=Z < \8,
\label{eq:indlimcon}
\end{align}
where the first limit is in $L_1$.
\end{theorem}

\begin{proof}
\textbf{Case 1: $x$ is the vertex of the cone.}

Without loss of generality, the vertex is assumed to be at the origin. The key observation is the following duality between the dynamics starting with $\cC$ and $\cC^{(0)}$,
\begin{align*}
H_{(0,t)} = \H_{(0,t)}^{(\Theta(\cC),0)}\circ V \circ S_{(0,t)}.
\end{align*}
Once this is established, the $L_1$ limit results from the fact that $S_{(0,t)}\circ V$ is measure preserving and therefore both sides are equivalent in distribution.

We first prove that 
$H_{(0,t)} \leq \H_{(0,t)}^{(\Theta(\cC),0)}\circ V \circ S_{(0,t)}$.
Consider the set of points
$(x_0,t_0), \ldots, (x_n,t_n) \in \R^d \times [0,t)$
``connecting'' $0$ with its height at $t$. Specifically, these satisfy:
\begin{enumerate}
\item $(x_i,t_i) \in \supp \Phi$ for $i = 0, \ldots, n$.
\item $0 \leq t_i < t_{i+1} < t$ for $i = 0, \ldots, (n-1)$.
\item $0 \in C_{(x_n,t_n)}$ and $H_{(0,s)} = H_{(0,t_n)}$ for $s \in [t_n,t]$.
\item There exists $y_i \in C_{(x_{i+1},t_{i+1})} \cap C_{(x_i,t_i)}$ such that $H_{(y_i,s)} = H_{y_i,t_i}$ for $s \in [t_i,t_{t+1})$ and $i = 0, \ldots, (n-1)$.
\item There exists $z \in C_{(x_0,t_0)} \cap \cC$ and $H_{(z,s)} = 0$ for $s \in [0,t_0)$.
\end{enumerate}
Let now $(\tilde x_i,\tilde t_i) = T_{(0,t)}R(x_{n-i},t_{n-i}) = (x_{n-i},t-t_{n-i})$ and $\tilde y_i = y_{n-i}$. Then, by the compatibility properties, these quantities satisfy:
\begin{enumerate}
\item $(\tilde x_i,\tilde t_i) \in \supp \Phi \circ V \circ S_{(0,t)}$ for $i = 0, \ldots, n$.
\item $0 \leq t_i < t_{i+1} < t$ for $i = 0, \ldots, (n-1)$.
\item $z \in C_{(\tilde x_n,\tilde t_n)}$ and $H_{(z,s)}^{(0)}\circ V \circ S_{(0,t)} = H_{(z,t_n)}^{(0)}\circ V\circ S_{(0,t)}$ for $s \in [\tilde t_n,t]$.
\item $\tilde y_{i+1} \in C_{(\tilde x_{i+1},\tilde t_{i+1})}\circ V\circ S_{(0,t)}\cap C_{(\tilde x_i,\tilde t_i)}\circ V\circ S_{(0,t)}$ such that $H_{(\tilde y_i,s)}^{(0)}\circ V\circ S_{(0,t)} = H_{\tilde x_i,\tilde t_i}^{(0)}\circ V\circ S_{(0,t)}$ for $s \in [\tilde t_i,\tilde t_{t+1})$ and $i = 0, \ldots, (n-1)$.
\item $0 \in C_{(\tilde x_1,\tilde t_1)}$ and $H_{(0,s)}^{(0)}\circ V \circ S_{(0,t)}= 0$ for $s \in [0,\tilde t_0)$.
\end{enumerate}
Given that $z \in \cC$ and $\cC$ is the convex cone $\cone$,
then for any $h>0$,
\begin{align*}
\frac{(z,h)}{|(z,h)|} \in \Theta(\cC).
\end{align*}
Then,
\begin{align*}
\H_{(0,t)}^{(\Theta(\cC),0)}\circ V \circ S_{(0,t)} &\geq H_{(z,t)}^{(0)}\circ V\circ S_{(0,t)}\\
&= \sum_{i=0}^n \s_{(\tilde x_i,\tilde t_i)}\circ V \circ S_{(0,t)}\\
&= \sum_{i=0}^n \s_{(x_i,t_i)}\\
&= H_{(0,t)}.
\end{align*}
The proof of the inequality in the other direction is similar to the previous one but starting with the dynamics of $\H_{(0,t)}^{(\Theta(\cC),0)} \circ V \circ S_{(0,t)}$.

\begin{figure}\label{fig:cone_base}
 \begin{center}
  \includegraphics[width=12cm]{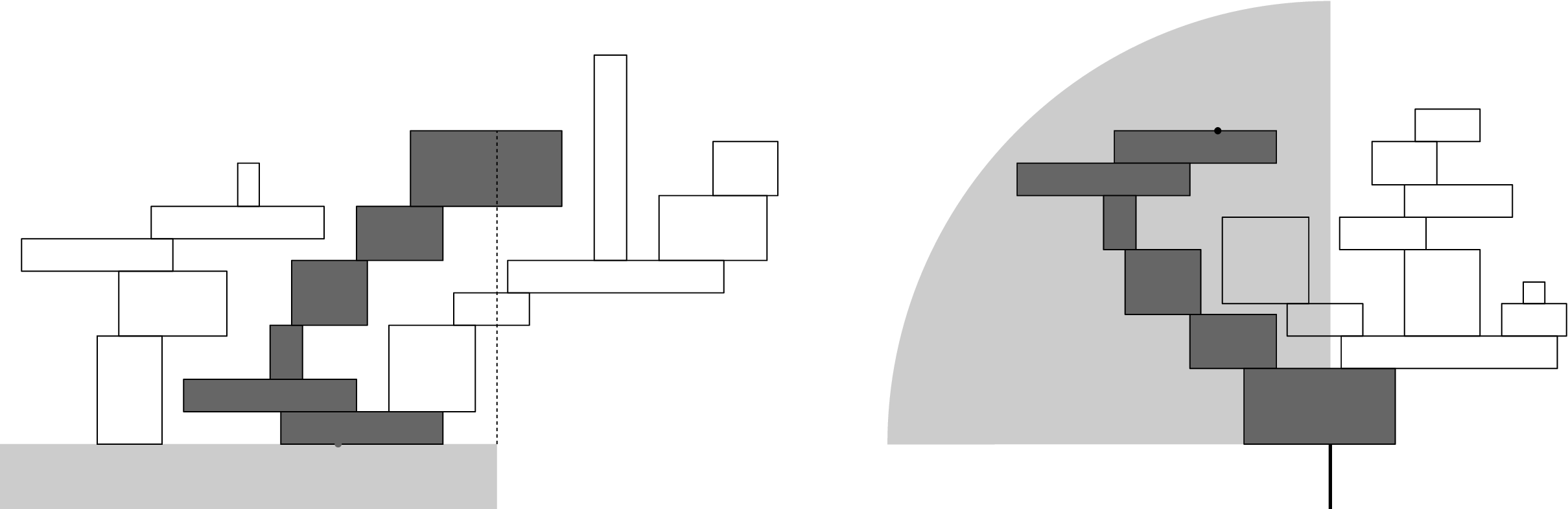}
 \end{center}
 \caption{Visualization of the duality argument in the proof of Theorem \ref{thm:cone_base} in the case $x=0$.}
\end{figure}

\textbf{Case 2: $x \in \cC$.}

The result in this case is obtained by comparison with the growth of the vertex studied in Case 1.
From the monotonicity w.r.t. the initial substrates,
\begin{align}\label{eq:33}
H_{(0,t)}^{(0)}\circ S_{(x,0)} \leq H_{(x,t)}.
\end{align}
On the other hand, using $\cC^{(x)} = \{0,x\}$, we have the following identity,
\begin{align}\label{eq:3}
H_{(x,t)} \leq \H^{(\Theta(\cC),1)}_t\circ V \circ S_{(0,t)}.
\end{align}
To prove the last relation, we use again a set of points $(x_0,t_0), \ldots, (x_n,t_n) \in \R^d \times [0,t)$ connecting $x$ with its height. As before:
\begin{enumerate}
\item $(x_i,t_i) \in \supp \Phi$ for $i = 0, \ldots, n$.
\item $0 \leq t_i < t_{i+1} < t$ for $i = 0, \ldots, (n-1)$.
\item $0 \in C_{(x_n,t_n)}$ and $H_{(0,s)} = H_{(0,t_n)}$ for $s \in [t_n,t]$.
\item There exists $y_i \in C_{(x_{i+1},t_{i+1})} \cap C_{(x_i,t_i)}$ such that $H_{(y_i,s)} = H_{y_i,t_i}$ for $s \in [t_i,t_{t+1})$ and $i = 0, \ldots, (n-1)$.
\item There exists $z \in C_{(x_0,t_0)} \cap \cC$ and $H_{(z,s)} = 0$ for $s \in [0,t_0)$.
\end{enumerate}
When we consider now $(\tilde x_i,\tilde t_i) = T_{(0,t)}\circ R(x_{n-i},t_{n-i}) = (x_{n-i},t-t_{n-i})$ and $\tilde y_i = y_{n-i}$, we then have that there exists a path of RACS starting at $x$ and finishing at $(z,H_{(x,t)}) \in \Theta(\cC)$. By the very definition of $\H^{(\Theta(\cC),1)}_t\circ V \circ S_{(0,t)}$ this then implies \eqref{eq:3}.

The desired limit then follows from (\ref{eq:33}) and (\ref{eq:3}), the result of Case 1 and Lemma \ref{lem:Cbounded2}. Notice that the $L_1$ limit is hence the same for all points $x\in \cC$.

\textbf{Case 3: $x \notin \cC$.}
Let $x\notin {\cC}$ and let $u$ be the Euclidean distance from $x$ to ${\cC}$, so $|x-y|=u$ for some
$y\in {\cC}$. On the segment $[x,y]$, choose points $x_0=y,x_1,\ldots,x_m=x$ equidistantly, where $m$ is the smallest integer that exceeds $2u/r$. Consider shifted versions of ${\cC}$, say ${\cC}_0=
{\cC}, {\cC}_1,\ldots , {\cC}_m$ such that for any $i\ge 1$, ${\cC}_i \supset {\cC}_{i-1}$, and ${\cC}_i$ includes the points $x_0,\ldots, x_i$ and does not include the points $x_{i+1},\ldots,x_m$. Then we show the convergence of $\max (0,H_{(x_i,t)})$ to $Z$ in $L_1$ using an induction argument: if the convergence holds for $x_{i-1}$, then it also holds for $x_i$. Because of that, we may assume without loss of generality that $m=1$, so $y=x_0$, $x=x_1$ and $|x-y|\le r/2$.

Let $\widetilde{\cC}= {\cC}_1$ and let $\widetilde{H}_{(y,t)}$ be the height associated with the cone $\widetilde{\cC}$. Let $\varepsilon$ be a positive number. 

First, we show that
\begin{equation}\label{3first}
\lim_{t\to\infty} {\mathbb P} \left(\frac{H_{(y,t)}}t > Z+\varepsilon \right) = 0
\end{equation}
and that the random variables $\max (0,H(y,t))/t$ are uniformly integrable. By monotonicity (see Subsection \ref{sec:23}), we have $H_{(y,t)} \le \widetilde{H}_{(y,t)}$ and, in view of the previous cases, $\max (0,\widetilde{H}_{(y,t)})/t \to Z$ in $L_1$ and, therefore, in probability. Therefore, both \eqref{3first} and uniform integrability of $H_{(y,t)}/t$ follow.

Secondly, we show that
\begin{equation}\label{3second}
\lim_{t\to\infty} {\mathbb P} \left(\frac{H_{(y,t)}}t < Z-\varepsilon \right)=0.
\end{equation}
Indeed, let $\Pi_{x,y}$ be a stream of RACS's that contain a ball of radius $r$ that covers both $x$ and $y$. By our assumptions, this is a homogeneous Poisson process of positive intensity, say $\nu$. For each $t$, let $t-\eta_t$ be the last arrival of such a RACS before $t$. Clearly, the random variable $\eta_t$ has an exponential distribution with parameter $\nu$. Further, $H_{(y,t)}\ge H_{(x,t-\eta_t)}$ a.s., so for any $T>0$,
$$
{\mathbb P} \left(\frac {H_{(y,t)}}t < Z-\varepsilon \right) 
\le {\mathbb P} (\eta_t >T) + {\mathbb P}\left( \frac{H_{(x,t-T)}}t <
Z-\varepsilon\right)
\to e^{-\nu T}
$$
as $t\to\infty$. Letting $T\to\infty$ leads to \eqref{3second}.

Finally equations \eqref{3first} and \eqref{3second} imply the convergence in probability $H_{(x,t)}/t \to Z$ and, further, uniform integrability implies the $L_1$-convergence of $\max (0, H_{(x,t)})$
to $Z$.

%
%
%
%
%
%
%
\end{proof}

\begin{definition}
We say that $\cC \ss \R^d$ is similar to the closed convex cone
$\cC^{(c)} \ss \R^d$ if there exists two vectors $v_\pm \in \R^d$ such that
\begin{align*}
\cC^{(c)} + v_- \ss \cC \ss \cC^{(c)} + v_+.
\end{align*}
\end{definition}

\begin{remark}
Notice that if $\cC$ is a convex cone, then it is trivially similar to itself. Also, if $\cC$ is similar to the convex cones $\cC^{(c)}_1$ and $\cC^{(c)}_2$ then $\cC^{(c)}_1 = \cC^{(c)}_2$ by the geometry of the convex cones.
\end{remark}

By monotonicity we obtain the following corollary from Theorem \ref{thm:cone_base}

\begin{corollary}
Assume the hypothesis of of Corollary \ref{cor:eventual_cover}
holds. Given $\cC \ss \R^d$ similar to a closed convex cone
$\cC^{(c)}$ with vertex
at the origin, for all $x \in \cC$,
\begin{align*}
\lim_{t\to\8} \frac{\max(H_{(x,t)},0)}{t} =
\lim_{t\to\8}\frac{\E \max( H_{(x,t)},0)}{t} =
\sup_{t>0}\frac{\E \H_t^{\1\Theta\1\cC^{(c)}\2,0\2}}{t} < \8,
\end{align*}
where the first limit is in $L_1$.
\end{corollary}

\section{Appendix}

\textit{Proof of Lemma \ref{lem:L1bound}}

The proof leverages the ideas developed in the proof of Theorem 2 in \cite{MR2865637}.
We use the same discretization of time and space
as in the proof of this theorem to show the following:
\begin{enumerate}
\item There exists a branching process constructed from an i.i.d. family of random variables
$\{(v_i,s_i)\}_i$ with light-tails. For a given $i$, $v_i$ denotes the number
of offsprings of $i$ and $s_i$ denotes the (common) {\em height} of its offspring.
\item 
For $n \in \N$, let $h(n)$ denote the maximum height of this branching process at generation $n$,
namely the maximum, over all lineages, of the sum of the heights of all generations in the lineage.
Then, in order to prove (\ref{eq:suphei}), it suffices to prove that $\E h(n) \leq Cn$ for every $n > 0$
for some finite $C$.
\end{enumerate}
For $n \in \N$, let $d_n$ denote the number of individuals of generation $n$ in this
branching process.  For $a>0$, let
\begin{align*}
D(a) := \bigcup_{n\geq 1}\{d_n > a^n\},\quad
\bar D(a) := \W \sm D(a).
\end{align*} 
Let $a_m = Cm$, $m \in \N$.
From Chernoff's inequality,
for some sufficiently large constant $C>0$,
\begin{align*}
\P(D(a_m)) \leq 2^{-m}.
\end{align*}
From Chernoff's inequality again, we get that
for all $i \in \N, \d > 0$ and $c_m > 0$ to be fixed we get,
\begin{align*}
\P\1\left\{\frac{h(n)}{n} > (c_m+i)\right\} \cap \bar D(a_m)\2 \leq \1a_m\E(e^{\d s})e^{-\d c_m}\2^ne^{-\d n i},
\end{align*}
where $s$ is a typical height. Therefore,
\begin{align*}
\E\1\frac{h(n)}{n}\2 &= \sum_{m\geq1} \E\1\frac{h(n)}{n}1_{\bar D(a_m) \sm \bar D(a_{m-1})}\2,\\
&\leq \sum_{m\geq1} \1\sum_{i\geq 0} \P\1\3\frac{h(n)}{n} > (c_m+i)\4 \cap \bar D(a_m)\2\2 + \P(\bar D(a_{m-1}))c_m,\\
&\leq \sum_{m\geq1} \1\frac{a_m\E(e^{\d s})e^{-\d c_m}}{1-e^{-\d}}\2^n + 2(2^{-m}c_m).
\end{align*}
Now we fix $\d$ sufficiently small such that $\E(e^{\d s}) < \8$. Recalling that $a_m = Cm$, in order
to conclude the proof, it suffices to construct $c_m$ independent of $n$, such that
\begin{align*}
\1Cme^{-\d c_m}\2^n \leq 2^{-m},\quad
\sum_{m\geq 1} 2^{-m}c_m < \8,
\end{align*}
where $C$ is a constant independent of $n$.
The last bound is satisfied for $c_m = Bm$ for any $B>0$.
However, for $B$ sufficiently large $Ce^{-\d Bm} \leq 4^{-m}$ which concludes the proof.


\bibliographystyle{plain}
\bibliography{mybibliography}

\end{document}